\documentclass[reqno,centertags, 12pt]{amsart}
\usepackage{comment}
\usepackage{color}

\usepackage{graphics}
\usepackage{amssymb,amsmath,amsfonts}
-\marginparwidth \oddsidemargin -9mm \evensidemargin \oddsidemargin
\usepackage{latexsym}
\advance\hoffset by 5mm

\def\@abssec#1{\vspace{.05in}\footnotesize \parindent .2in
{\bf #1. }\ignorespaces}
\newtheorem{theorem}{Theorem}[section]

\newtheorem{lemma}[theorem]{Lemma}
\newtheorem{proposition}[theorem]{Proposition}

\newcommand{\dint}{\displaystyle\int}

\newcommand{\be}{\mathbf e}

\newcommand{\lb}{\label}

\allowdisplaybreaks \numberwithin{equation}{section}

\renewcommand{\be}{\begin{equation}}
\newcommand{\ee}{\end{equation}}
\newcommand{\rf}{\eqref}
\newcommand{\om}{\omega}
\newcommand{\OO}{\mathcal O}

\title[Small scale creation for 2D Euler]{Small scale creation for solutions of the incompressible two dimensional Euler equation}
\author{Alexander Kiselev}
\thanks{Department of
Mathematics, Rice University, 6100 Main MS-134, Houston TX 77005, USA;\\
email: kiselev@rice.edu}
\author{Vladimir $\check{{\rm S}}$ver\'ak}
\thanks{Department of
Mathematics, University of Minnesota, 206 Church St. SE, Minneapolis, MN 55455, USA;
email: sverak@math.umn.edu}

\begin{document}


\begin{abstract}
We construct an initial data for the two-dimensional Euler equation in a disk for which the gradient of vorticity
exhibits double exponential growth in time for all times. This estimate is known to be sharp - the double exponential growth is the fastest possible growth rate.
\end{abstract}


\maketitle

\section{Introduction}\label{intro}

The two dimensional Euler equation for the motion of an inviscid, incompressible fluid is given in vorticity form by
\begin{equation}\label{2de}
\partial_t \omega + (u \cdot \nabla) \omega =0, \,\,\, \omega(x,0)=\omega_0(x).
\end{equation}
Here $\omega$ is the vorticity of the flow, and the fluid velocity $u$ is determined from $\omega$ by the appropriate Biot-Savart law.
If we consider fluid in a smooth bounded domain $D,$ we impose a no flow condition at the boundary: $u(x,t) \cdot n(x) =0$ for $x \in \partial D.$
This implies that $u(x,t) = \nabla^\perp \int_D G_D(x,y) \omega(y,t)\,dy,$ where $G_D$ is the Green's function for the Dirichlet problem in $D$ and
it will be convenient for us to follow the convention that $\nabla^\perp = (\partial_{x_2}, -\partial_{x_1}).$
We will denote $K_D(x,y) = \nabla^\perp G_D(x,y),$ so that $u(x,t) = \int_D K_D(x,y) \omega(y,t)\,dy.$

The global regularity of solutions to two-dimensional Euler equation is known since the work of Wolibner \cite{Wolibner} and H\"older \cite{Holder},
see also \cite{Kato2}, \cite{Koch}, \cite{MP} or \cite{Chemin}
for more modern and accessible proofs. The two-dimensional Euler equation is critical in the sense that the estimates needed to obtain global regularity
barely close. The best known upper bound on the growth of the gradient of vorticity and higher order Sobolev norms is double exponential
in time. This result is well known and has first appeared in \cite{YudDE}, though related bounds can be traced back to \cite{Wolibner} and \cite{Holder}.
We will sketch an argument obtaining the double exponential bound below in Section~\ref{prelim} for the sake of completeness.

The question of whether such upper bounds are sharp has been open for a long time. Yudovich \cite{Jud1,Yud2} provided an example showing infinite
growth of the vorticity gradient at the boundary of the domain, by constructing an appropriate Lyapunov functional. These results were further
improved and generalized in \cite{MSY}, leading to description of a broad class of flows with infinite growth in their vorticity gradient.
Nadirashvili \cite{Nad1} proved a more quantitative linear in time lower bound for a ``winding" flow in an annulus. Bahouri and Chemin \cite{BC} provided an example of singular stationary
solution of the 2D Euler equation which produces a flow map whose H\"older regularity decreases in time. This example also has a fluid velocity which
is just log-Lipschitz in the spatial variables, the lack of Lipschitz regularity that is exactly related to the possibility of double exponential growth.
We refer to \cite{Koch} and \cite{MSY} for a more thorough discussion of the history of the problem and connections with classical stability questions.

In recent years, there has been a series of works by Denisov on this problem. In \cite{Den1}, he constructed an example with superlinear growth in its vorticity gradient
in the periodic case. In \cite{Den2}, he showed that the growth can be double exponential for any given (but finite) period of time.
In \cite{Den3}, he constructed a patch solution to 2D Euler equation where, under the action of a regular prescribed stirring velocity, the distance between the boundaries of two
patches decreases double exponentially in time. We also refer to a discussion at Terry Tao's blog \cite{Tao} for more information on the problem
and related questions.

Our goal in this paper is to construct an example of initial data in the disk such that the corresponding solution for the 2D Euler equation
exhibits double exponential growth in the gradient of vorticity. We do not require any force or controlled stirring in the equation. Namely, we will prove
\begin{theorem}\label{mainthm}
Consider two-dimensional Euler equation on a unit disk $D.$ There exists a smooth initial data $\omega_0$ with 
$\|\nabla \omega_0\|_{L^\infty}/\|\omega_0\|_{L^\infty} >1$ such that the corresponding
solution $\omega(x,t)$ satisfies
\begin{equation}\label{main1}
\frac{\|\nabla \omega(x,t)\|_{L^\infty}}{\|\omega_0\|_{L^\infty}} \geq \left( \frac{\|\nabla \omega_0\|_{L^\infty}}{\|\omega_0\|_{L^\infty}}\right)^{c\exp(c\|\omega_0\|_{L^\infty}t)}
\end{equation}
for some $c>0$ and for all $t \geq 0.$
\end{theorem}
We also note that our construction yields exponential in time growth for $\|\nabla u \|_{L^\infty},$ thus answering a question asked by Koch in \cite{Koch}.

The theorem shows that the double exponential upper bound is in general optimal for the growth of vorticity gradient of solutions to the two-dimensional Euler equation.
The growth in our example happens at the boundary. We do not know if such growth is possible in the bulk of the fluid. In a few lines, the idea of the construction
can be described as follows. The Bahouri-Chemin solution in the periodic case can be thought of as a singular ``cross" stationary solution satisfying $\omega(x_1,x_2)=1$ for $0 \leq x_1,x_2 \leq 1,$
odd with respect to both coordinate axes, and periodic with period $2$ in each direction. The lines $x_1=0$ and $x_2=0$ are separatices for this flow, and the origin is a
hyperbolic point. If one considers the fluid velocity created by the cross, one can check that $u_1(x_1,0) \sim x_1 \log x_1$ for small $x_1.$ A passive scalar advected by such a flow
will generally exhibit double exponential growth of the gradient. If one could smooth out this flow, arrange for a small perturbation of it to play the role of a passive scalar on top of
singular behavior, and somehow arrange for the solution to approach the singular ``cross'' solution of the background flow, then one could provide an example of double exponential in time growth.

A similar idea was exploited in \cite{Den2} to design a finite time double exponential growth example. However, one would face serious difficulties to extend this approach to infinite time.
First, to keep the background
scenario stable, one needs symmetry - and odd symmetry bans nonzero perturbation right where the velocity is most capable of producing double exponential growth for all times, on the $x_2=0$ separatrix.
Second, it is not clear how to make a smooth solution approach the ``cross" in some suitable sense. Third, the perturbation will not be passive, and, in large time limit, will be difficult to decouple
from the equation. Our observation is that one can use the boundary to avoid dealing with the first issue and to make the second issue more manageable.
In this sense, our work goes back to the idea of Yudovich that boundaries are prone to the generation of small scales in solutions of 2D Euler equation.

\section{Preliminaries}\label{prelim}

Here we collect some background information for the sake of completeness. First, we will sketch a proof of the following theorem. The result goes back to \cite{Wolibner}, \cite{Holder}.
Throughout the paper, $C$ will denote universal constants that may change from line to line.

\begin{theorem}\label{upperb}
Let $D$ be a bounded domain with smooth boundary. Let $\omega_0(x)$ be smooth initial data for the 2D Euler equation.
Then the solution $\omega(x,t)$ satisfies
\begin{equation}\label{upperdexp}
1+\log\left(1+\frac{\|\nabla \omega(x,t)\|_{L^\infty}}{\|\omega_0\|_{L^\infty}}\right) \leq
\left(
1+\log\left(
1+
{
\|\nabla\omega_0\|_{L^\infty}
\over
\|\omega_0\|_{L^\infty}
}
\right)
\right)
\exp(
C\|\omega_0\|_{L^\infty}t
)
\end{equation}
for some constant $C$ which may depend only on the domain $D.$
\end{theorem}

We will need a potential theory bound
\begin{proposition}\label{kato}
Let $u,\omega$ be velocity and vorticity solving 2D Euler equation in domain $D.$ Then for every $1>\alpha>0,$ we have
\begin{equation}\label{katoeq}
\|\nabla u(x,t)\|_{L^\infty} \leq C(\alpha,D) \|\omega_0\|_{L^\infty}\left(1+\log \frac{\|\omega(x,t)\|_{C^\alpha}}{\|\omega_0\|_{L^\infty}} \right).
\end{equation}
\end{proposition}
Here and below we will use the convention that $||\om(x,t)||$ denotes the norm of the function $x\to \om(x,t)$, i.\ e.\ we view $\om$ as a function of $x$ while $t$ is viewed as a parameter. Also, we use the standard notation
$\|f(x)|_{C^\alpha}=\|f(x)\|_{L^\infty}+\sup_{x\ne y} |f(x)-f(y)|/|x-y|^\alpha.$

Since our domain $D$ is fixed, we will work with a fixed length-scale and therefore we can use expressions
such as $||\nabla\om(x,t)||_{L^\infty}/||\om_0||_{L^\infty}$ instead of the ``dimensionally balanced" expression $L ||\nabla\om(x,t)||_{L^\infty}/||\om_0||_{L^\infty}$, where $L$ is a reference length. Our reference length is simply taken to be $L=1$.

The bound \eqref{katoeq} was first proved by Kato \cite{Kato} in the whole plane case, and its three dimensional analog is a key component in the Beale-Kato-Majda criterion
for 3D Euler regularity \cite{BKM}.
Given Proposition~\ref{kato}, Theorem~\ref{upperdexp} follows.
\begin{proof}[Proof of Theorem~\ref{upperdexp}]
Let us denote by $\Phi_t(x)$ the flow map corresponding to the 2D Euler evolution:
\[ \frac{d\Phi_t(x)}{dt} = u(\Phi_t(x),t), \,\,\,\Phi_0(x)=x. \]
Then
\[ \left| \frac{\partial_t |\Phi_t(x)-\Phi_t(y)|}{|\Phi_t(x)-\Phi_t(y)|} \right| \leq \|\nabla u\|_{L^\infty} \leq \\
 C \|\omega_0\|_{L^\infty}\left(1+\log \left( 1+ \frac{\|\nabla \omega(x,t)\|_{L^\infty}}{\|\omega_0\|_{L^\infty}} \right) \right). \] 
After integration, this gives
\begin{equation}\label{twoway}
f(t)^{-1} \leq \frac{|\Phi_t(x)-\Phi_t(y)|}{|x-y|} \leq f(t),
\end{equation}
where
\[ f(t) = \exp \left(C \|\omega_0\|_{L^\infty}\int_0^t  \left(1+\log \left( 1+ \frac{\|\nabla \omega(x,s)\|_{L^\infty}}{\|\omega_0\|_{L^\infty}} \right) \right)\,ds\right). \]
Of course, the bound \eqref{twoway} also holds for $\Phi_t^{-1}.$
On the other hand,
\begin{equation}\label{vorteq}
\|\nabla \omega(x,t)\|_{L^\infty} = {\rm sup}_{x,y} \frac{|\omega_0(\Phi_t^{-1}(x))-\omega_0(\Phi_t^{-1}(y))|}{|x-y|} \leq
\|\nabla \omega_0\|_{L^\infty} {\rm sup}_{x,y} \frac{|\Phi_t^{-1}(x)-\Phi_t^{-1}(y)|}{|x-y|}.
\end{equation}
Combining \eqref{vorteq} and \eqref{twoway}, we obtain
\[ \|\nabla \omega(x,t)\|_{L^\infty} \leq \|\nabla \omega_0\|_{L^\infty} \exp \left(C  \|\omega_0\|_{L^\infty} \int_0^t \left(1+\log \left(1+ \frac{\|\nabla \omega(x,s)\|_{L^\infty}}{\|\omega_0\|_{L^\infty}} \right) \right)\,ds\right), \]
or
\[ \log \|\nabla \omega(x,t)\|_{L^\infty} \leq \log \|\nabla \omega_0\|_{L^\infty} + C \|\omega_0\|_{L^\infty}\int_0^t \left(1+\log \left(1+  \frac{\|\nabla \omega(x,s)\|_{L^\infty}}{\|\omega_0\|_{L^\infty}} \right) \right)\,ds. \]

Let
$
A=||\omega_0||_{L^\infty}\,,\,\, B=||\nabla\omega_0||_{L^\infty}\,
$
and consider the solution $y=y(t)$ of
\be\lb{z1}
{y'\over y} =C\!A\left(1+\log(1+y)\right)\,,\qquad y(0)={B\over A}=y_0\,.
\ee
By Gronwall's lemma it is enough to bound $y(t)$.
The solution of~\rf{z1} is given by
\be\lb{z2}
\int_{y_0}^{y(t)} {dy\over y\left(1+\log(1+y)\right)} = C\!At\,.
\ee
Hence
\begin{align*}
 & \log\left(1+\log(1+y(t))\right)-\log\left(1+\log(1+y_0)\right) \\ & \qquad +\int_{y_0}^{y(t)}
dy \,\left[{1\over y(1+\log(1+y))}-{1\over(1+y)(1+\log(1+y))}\right]=C\!At\,.
\end{align*}
The integrant in the last expression is positive and hence
\be\lb{z4}
1+\log(1+y(t))\le(1+\log (1+y_0)) \exp(C\!At)\,.
\ee
\end{proof}
Now we sketch the proof of Proposition~\ref{kato}.
\begin{proof}[Proof of Proposition~\ref{kato}]
Let \[ \delta = {\rm min}\left(c, \left(\frac{\|\omega_0\|_{L^\infty}}{\|\omega(x,t)\|_{C^\alpha}}\right)^{1/\alpha}\right),\] where $1>c>0$ is some fixed constant that depends on $D,$
chosen so that the set of points
$x \in D$ with ${\rm dist}(x,\partial D) \geq 2\delta$ is not empty. 
Consider first any interior point $x$ such that ${\rm dist}(x, \partial D) \geq 2\delta.$
Standard computations (see e.g. \cite{MB}) show that
\[ \nabla u(x) = P.V. \int_D \nabla K_D(x,y) \omega(y)\,dy + \frac{(-1)^{i}}{2} \omega(x)(1-\delta_{ij}). \]
The part of the integral over the complement of the ball centered at $x$ with radius $\delta$ can be estimated as
\begin{equation}\label{outpart1} \left| \int_{B_\delta^c(x)} \nabla K_D(x,y) \omega(y)\,dy \right| \leq C\|\omega_0\|_{L^\infty}
\int_{B_\delta^c(x)} |x-y|^{-2} \,dy \leq C \|\omega_0\|_{L^\infty} (1+\log \delta^{-1}), \end{equation}
where we used a well known bound $|\nabla K_D(x,y)| \leq C|x-y|^{-2}.$

Now recall that the Dirichlet Green's function is given by
\begin{equation}\label{greenf}
G_D(z,y)= \frac{1}{2\pi}\log|z-y| + h(z,y),
\end{equation}
where $h$ is harmonic in $D$ in $z$ for each fixed $y$ and has boundary values $-\frac{1}{2\pi}\log|z-y|.$
Any second order partial derivative at $z=x$ of the first summand on the right hand side of \eqref{greenf} is of the form $r^{-2}\Omega(\phi)$ where $r,\phi$ are radial variables centered at $x,$ and $\Omega(\phi)$ is mean zero.
For this part, we can write
\begin{eqnarray}
\nonumber \left| P.V. \int_{B_\delta(x)} \partial^2_{x_i x_j} \log |x-y| \omega(y)\,dy \right| = \left| \int_{B_\delta(x)} \partial^2_{x_i x_j} \log |x-y| (\omega(y)-\omega(x))\,dy \right| \\
\leq C \|\omega(x,t)\|_{C^\alpha} \int_0^\delta r^{-1+\alpha} \,dr \leq C(\alpha) \delta^\alpha \|\omega(x,t)\|_{C^\alpha} \leq C(\alpha)\|\omega_0\|_{L^\infty}   \label{inone}
\end{eqnarray}
by our choice of $\delta.$
Finally, notice that our assumptions on $x$, our choice of boundary values for $h$, and the maximum principle together guarantee that
we have $|h(z,y)| \leq C \log \delta^{-1}$ for all $y \in
B_{\delta}(x),$ $z \in D.$
Standard estimates for harmonic functions (see e.g. \cite{Evans}) give, for each fixed $y \in B_\delta(x),$
\[ |\partial^2_{x_i x_j}h(x,y) | \leq C \delta^{-4} \|h(z,y)\|_{L^1(B_\delta(x),dz)} \leq C \delta^{-2} \log \delta^{-1}. \]
This gives
\begin{equation}\label{intwo}
\left| \int_{B_\delta(x)} \partial^2_{x_i x_j} h(x,y) \omega(y,t)\,dy \right| \leq C \|\omega_0\|_{L^\infty} \log \delta^{-1}.
\end{equation}
Together, \eqref{intwo}, \eqref{inone} and \eqref{outpart1} prove the Proposition at interior points.

Now if $x'$ is such that ${\rm dist}(x', \partial D) < 2\delta,$ find a point $x$ such that ${\rm dist}(x, \partial D) \geq 2\delta$ and $|x'-x| \leq C(D)\delta.$
By the following Schauder estimate (see e.g. \cite{GT}) we have
\begin{equation}\label{Schauder} |\nabla u(x') - \nabla u(x)| \leq C(\alpha,D) \delta^\alpha \|\omega\|_{C^\alpha}. \end{equation}
At $x,$ interior bounds apply, which together with \eqref{Schauder} gives desired bound at any $x' \in D.$
\end{proof}

\section{The Key Lemma}\label{keylemma}

As the first step towards the proof of Theorem~\ref{mainthm}, let us start setting up the scenario we will be considering. From now on, let $D$ be a closed unit disk in the plane. It will be convenient for
us to take the system of coordinates centered at the lowest point of the disk, so that the center of the disk is at $(0,1).$ Our initial data $\omega_0(x)$ will be odd with respect to the vertical
axis: $\omega_0(x_1,x_2) = -\omega_0(-x_1, x_2).$ It is well known and straightforward to check that 2D Euler evolution preserves this symmetry for all times.

We will take smooth initial data $\omega_0(x)$ so that $\omega_0(x) \geq 0$ for $x_1 >0$ (and so $\omega_0(x) \leq 0$ for $x_1<0$). This configuration 
makes the origin a hyperbolic fixed point of the flow; in particular, $u_1$ vanishes on the vertical axis. Let us analyze the Biot-Savart law we have for the disk to gain insight into the structure of the velocity field. The Dirichlet Green's
function for the disk is given explicitly by
$G_D(x,y) = \frac{1}{2\pi} (\log |x-y| - \log |x-\bar{y}| -\log |y-e_2|),$ where with our choice of coordinates $\bar{y}= e_2+ (y-e_2) / |y-e_2|^2,$ $e_2=(0,1).$ Given the symmetry of $\omega,$
we have
\begin{equation}\label{BioS23}
u(x,t) = \nabla^\perp \int_D G_D(x,y)\omega(y,t)\,dy =  \frac{1}{2\pi} \nabla^\perp \int_{D^+}  \log\left(\frac{|x-y||\tilde{x}-\bar{y}|}{|x-\bar{y}||\tilde{x}-y|}\right) \omega(y,t)\,dy,
\end{equation}
where $D^+$ is the half disk where $x_1 \geq 0,$ and $\tilde{x} = (-x_1,x_2).$ The following Lemma will be crucial for the proof of Theorem~\ref{mainthm}.
Let us introduce notation $Q(x_1,x_2)$ for a region that is the intersection of $D^+$ and the quadrant $x_1 \leq y_1 < \infty,$ $x_2 \leq y_2 < \infty.$

\begin{lemma}\label{mainlemma}
Take any $\gamma,$ $\pi/2>\gamma>0.$ Denote $D^\gamma_1$ the intersection of $D^+$ with a sector $\pi/2-\gamma \geq \phi \geq 0,$ where $\phi$ is the usual angular variable.
Then there exists $\delta >0$ such that for all $x \in D^\gamma_1$ such that $|x| \leq \delta$ we have
\begin{equation}\label{velest1}
u_1(x_1,x_2,t) = - \frac{4}{\pi} x_1\int_{Q(x_1,x_2)} \frac{y_1y_2}{|y|^4} \omega(y,t)\,dy_1dy_2 + x_1 B_1(x_1,x_2,t),
\end{equation}
where $|B_1(x_1,x_2,t)| \leq C(\gamma)\|\omega_0\|_{L^\infty}.$

Similarly, if we denote $D^\gamma_2$ the intersection of $D^+$ with a sector $\pi/2\geq \phi \geq \gamma,$ then for all $x \in D^\gamma_2$ such
that $|x| \leq \delta$ we have
\begin{equation}\label{velest2}
u_2(x_1,x_2,t) = \frac{4}{\pi}x_2\int_{Q(x_1,x_2)} \frac{y_1y_2}{|y|^4} \omega(y,t)\,dy_1dy_2 + x_2 B_2(x_1,x_2,t),
\end{equation}
where $|B_2(x_1,x_2,t)| \leq C(\gamma)\|\omega_0\|_{L^\infty}.$
\end{lemma}

\noindent
\it Remarks. \rm 1. The Lemma holds more generally than in the disk; perhaps the simplest proof is for the case where $D$ is a square. \\
2. The exclusion of a small sector does not appear to be a technical artifact. The vorticity can be arranged (momentarily) in a way that the hyperbolic picture provided by the
Lemma is violated outside of $D^\gamma_1$, for example the direction of $u_1$ may be reversed near the vertical axis.

Essentially, Lemma~\ref{mainlemma} makes it possible to ensure that the flow near the origin is hyperbolic, with fluid trajectories just hyperbolas in the main term.
The speed of motion along trajectories is controlled by the nonlocal factor in \eqref{velest1}, \eqref{velest2} (denoted $\Omega(x_1,x_2,t)$ below), and this factor
is the same for both $u_1$ and $u_2.$
\begin{proof}
 The proof of the lemma can be carried out in a number of ways. Here we present a version which uses the scaling symmetry of the estimates. For $R>0$ we let $D_R=R\, D=\{Rx,\,x\in D\}$ be the disc of radius $R$ centered at $(0,R)$. We will also use the notation $D^+_R=R\, D^+$ and $D^\gamma _R=R\, D^\gamma_1$.
  We denote by $\psi$ the stream function generated by $\om$ (via the equation $\Delta \psi =\omega$ and the boundary condition $\psi|_{\partial D}=0$). For $x\in D^+$ we set
 \be\lb{e1}
 q_1(x)={u_1\over x_1}=\frac{1}{x_1}\frac{\partial \psi}{\partial x_2},\qquad q_2={u_2\over x_2}=-\frac{1}{x_2} \frac{\partial\psi}{\partial x_1} \,.
 \ee
 These quantities are of course time-dependent, but our estimates will be proved at each fixed time and therefore the dependence of the quantities on $t$ does not have to be indicated.
 For $\lambda>0$ we consider the scaling transformation
 \be\lb{e2}
 \om(x)\to \om \left({x\over\lambda}\right)\,,\quad \psi(x)\to \lambda^2\psi\left({x\over\lambda}\right)\,,\quad q_j(x)\to q_j\left({x\over \lambda}\right)\,\,,\quad j=1,2.
 \ee
 Using the transformation, we see that instead of proving~\rf{velest1} in $D_1^\gamma$, we can prove instead the same estimate in $D_R^\gamma$ (for the new quantities obtain by the re-scaling) under an additional assumption $x_1=1$, as long as the estimate will be independent of $R\ge 1$. We still use the notation $\om,\psi,q_1, q_2$ for the re-scaled quantities. Also, the inversion with respect to the disc $D_R$ will still be denoted as $y\to \bar y$, with a slight abuse of notation. (It would be more precise to use the more cumbersome $\bar y_R$ instead of $\bar y$.)

 We have
 \be\lb{e3}
 \psi(x)=\dint_{D^+_R} G(x,y)\om(y)\,dy\,,
 \ee
with $G=G_R$ given by
\be\lb{e4}
2\pi G(x,y)=\log|x-y|-\log|\tilde x-y|+\log |\tilde x-\bar y|-\log |x-\bar y|\,,\qquad \bar y =\bar y_R\,.
\ee
Let us prove estimate~\rf{velest1} in $D^+_R$, while assuming $x_1=1$. We  fix $r>0$ such that
\be\lb{e6}
r^2\ge 100(1+\tan^2 \gamma)
\ee
and will assume
\be\lb{e7}
R\ge 10 r\,.
\ee
As it is clearly enough to show the original estimate only for small $x_1$, condition \rf{e7} can be assumed without loss of generality.
We will use the usual notation $B_r=\{x\,,\,\,|x|<r\}$.

We note that the contribution to $q_1$ from the region $D_R^+\cap B_r$ is bounded by standard elliptic estimates; this can be also shown by
a direct computation using explicit formula \eqref{BioS23}.
In fact, the estimate we need is elementary, it essentially reduces to $\int_{B_r} {dx\over |x|}\le C(r)$.
We can therefore assume without loss of generality that $\om$ is supported in $D^+_R\setminus B_r$. With this assumption the calculation reduces to finding a suitable expansion of $G(x,y)=G_R(x,y)$ for large $y$.
We note that
\begin{eqnarray*}
4\pi G(x,y) & = &
\,\,\,\,\,\log\left(1-{2xy\over |y|^2}+{|x|^2\over |y|^2}\right)
-\log\left(1-{2\tilde x y\over |y|^2}+{|x|^2\over |y|^2}\right) \\
& &  -\log
\left(
1-{2x\bar y\over |\bar y|^2}
+{|x|^2\over |\bar y|^2}\right)+
\log\left(
1-{2\tilde x{\bar y}\over |\bar y|^2}
+{|x|^2\over |\bar y|^2}\right)\,\,.\\
\end{eqnarray*}
Using
\be\lb{e10}
\log(1+t)=t-{t^2\over 2}+{t^3\over 3} \varkappa(t)\,,
\ee
where $\varkappa(t)$ is an analytic function with $\varkappa(0)=1$,
we easily check that
\be\lb{e11}
\pi  G(x,y)= -{x_1y_1\over |y|^2}+{x_1\bar y_1\over |\bar y|^2} -2{x_1x_2y_1y_2\over |y|^4}+2{x_1x_2\bar y_1\bar y_2\over |\bar y|^4}+O\left({1\over |y|^3}+{1\over |\bar y|^3}\right)\,.
\ee
Here and also in the expressions below the implied constants in the $O-$notation can depend on~$\gamma$, but are uniform in $x$ once $\gamma$ is fixed and $x$ lies on the segment $x_1=1, x_2\le \tan \gamma, x\in D^+_R$.
We  calculate
\be\lb{e12}
{\bar y_1\over |\bar y|^2}
={y_1\over |y|^2}
\,,\qquad
{\bar y_2\over |\bar y|^2}=-{y_2\over |y|^2}+{1\over R}\,,
\ee
which gives
\be\lb{e13}
\pi G(x,y)=-4{x_1x_2y_1y_2\over |y|^4}+2{x_1x_2 y_1\over |y|^2 R}+O\left({1\over |y|^3}+{1\over |\bar y|^3}\right).
\ee
We can also easily check that, denoting by $S$ the segment $x_1=1, x_2\le \tan \gamma, x\in D_R$, we have
\be\lb{e14}
\pi {\partial G(x,y)\over \partial x_2}|_{x\in S}=-4{y_1y_2\over |y|^4} + 2{y_1\over |y|^2R}+ O\left({1\over |y|^3}+{1\over |\bar y|^3}\right).
\ee
In addition, it is easy to see that for $y\in D_R$ we have $|\bar y|\ge |y|$.
The proof of~\rf{velest1} now follows from the following elementary estimates:
\begin{eqnarray*}
\dint_{D_R^+\setminus B_r} \left({1\over |y|R}+ {1\over |y|^3}+{1\over |\bar y|^3} \right)\,dy & = & O(1), \qquad R\to\infty\,\,,\\
\dint_{(D^+_R\setminus Q(x))\cap(D^+_R\setminus B_r)} {y_1y_2\over |y|^4}\,\,dy & = & O(1) ,\qquad R\to\infty\,\,.
 \end{eqnarray*}
 The proof of~\rf{velest2} is  the same, except that we re-scale to  $x_2=1$ and replace~\rf{e14} by the corresponding expression for $-\pi{\partial G\over \partial x_1}$.

\end{proof}

Before proving Theorem~\ref{mainthm}, we make a simpler observation: with the aid of Lemma~\ref{mainlemma} it is fairly straightforward to find examples with exponential in time growth
of vorticity gradient. Indeed, take smooth initial data $\omega_0(x)$ which is equal to one everywhere in $D^{+}$ except on a thin strip of width equal to $\delta$ near the vertical axis $x_1=0$,
where $0<\omega_0(x)<1$ (and $\omega_0$ vanishes on the vertical axis as it must by our symmetry assumptions). Observe that due to
incompressibility, the distribution function of $\omega(x,t)$ is the same for all times. In particular,
the measure of the complement of the set where $\omega(x,t)=1$ does not exceed $2\delta$.
In this case for every $|x| <\delta,$ $x \in D^+,$ we can derive the following estimate for the integral appearing in the representation \eqref{velest1}:
\[ \int_{Q(x_1,x_2)} \frac{y_1y_2}{|y|^4} \omega(y,t)\,dy_1dy_2 \geq  \int_{2\delta}^1 \int_{\pi/6}^{\pi/3} \omega(r,\phi) \frac{\sin 2\phi}{2r}\, d\phi dr\geq
\frac{\sqrt{3}}{4}\int_{2\delta}^1 \int_{\pi/6}^{\pi/3} \frac{\omega(r,\phi)}{r} \,d\phi dr. \]
The value of the integral on the right hand side is minimal when the area where $\omega(r,\phi)$ is less than one is concentrated at small values of the radial variable.
Using that this area does not exceed $2\delta,$ we obtain
\begin{equation}\label{mainterm111}
\frac{4}{\pi} \int_{Q(x_1,x_2)} \frac{y_1y_2}{|y|^4} \omega(y,t)\,dy_1dy_2 \geq c_1 \int_{c_2 \sqrt{\delta}}^1 \int_{\pi/6}^{\pi/3} \frac{1}{r} \,d\phi dr \geq C_1 \log \delta^{-1},
\end{equation}
where $c_1,$ $c_2$ and $C_1$ are positive universal constants.

Putting the estimate \eqref{mainterm111} into \eqref{velest1}, we get that for all for $|x| \leq \delta,$ $x \in D^+$ that lie on the disk boundary, we have
\[ u_1(x,t) \leq -x_1(C_1 \log \delta^{-1} -C_2), \]
where $C_{1,2}$ are universal constants. We can choose $\delta>0$ sufficiently small so that $u_1(x,t) \leq -x_1$ for all times if $|x|<\delta.$
Due to the boundary condition on $u$, the trajectories which start at the boundary stay on the boundary for all times.
Taking such a trajectory starting at a point $x_0 \in \partial D$ with $x_{0,1} \leq \delta$, we get $\Phi^1_{t,1}(x_0) \leq x_{0,1} e^{-t}$ for this characteristic curve.
Since $\omega(x,t) = \omega(\Phi^{-1}_t(x)),$ we see that $\|\nabla \omega(x,t)\|_{L^\infty}$ grows exponentially in time if we pick $\omega_0$ which does not vanish identically
at the boundary near the origin (for example, if $\omega_0(\delta,1-\sqrt{1-\delta^2}) \ne 0$).

\section{The Proof of the Main Theorem}\label{main}

To construct examples with double exponential growth, we have to work a little harder. For the sake of simplicity, we will build our example with $\omega_0$ such that $\|\omega_0\|_{L^\infty}=1.$
\begin{proof}[Proof of Theorem~\ref{mainthm}]
We first fix some small $\gamma >0.$
Denote
\[ \Omega(x_1, x_2, t) = \frac{4}{\pi}\int_{Q(x_1,x_2)} \frac{y_1y_2}{|y|^4}\omega(y,t)\,dy_1dy_2. \]
We will take the smooth initial data like in the end of the previous section, with $\omega_0(x)=1$ for $x \in D^+$ apart from a narrow strip of width at most $\delta>0$ (with $\delta$ small) near the vertical axis where $0 \leq \omega_0(x) \leq 1.$
Then \eqref{mainterm111} holds. We will also choose $\delta$ so that $C_1 \log \delta^{-1}>100 C(\gamma)$ where $C(\gamma)$ is the constant
in the bound for the error terms $B_1,$ $B_2$ appearing in \eqref{velest1}, \eqref{velest2}.

For $0<x_1'<x_1''<1$ we  denote
\be\lb{xx1}
\OO(x_1',x_1'')=\left\{
(x_1,x_2)\in D^+\,,\,\,x_1'<x_1<x_1''\,,\,\,
x_2<x_1\right\}\,.
\ee
For $0<x_1<1$ we let
\be\lb{xx2}
{\underline u}_1(x_1,t) \quad =\quad  \min_{(x_1,x_2)\in D^+\,,\,x_2<x_1} u_1(x_1,x_2,t)
\ee
and
\be\lb{xx3}
\,\,{\overline u}_1(x_1,t) \quad = \quad \max_{(x_1,x_2)\in D^+\,,\, x_2<x_1} u_1(x_1,x_2,t)\,.
\ee
It is easy to see that these functions are locally Lipschitz in $x_1$ on $[0,1)$, with the Lipschitz constants  being locally bounded in time. Hence we can define
$a(t)$  by
\be\lb{xx4}
\dot a= \overline u_1(a,t)\,,\quad a(0)=\epsilon^{10}\,
\ee
and  $b(t)$  by
\be\lb{xx5}
\dot b = \underline u_1(b,t)\,,\quad b(0)=\epsilon\,,
\ee
where
$0<\epsilon< \delta$ is sufficiently
small, its exact value to be determined later.
Let
\be\lb{xx6}
\OO_t=\OO(a(t),b(t))\,.
\ee
At this stage we have not yet ruled out that $\OO_t$ perhaps might become empty for some $t>0$. However, it is clear from the definitions that $\OO_t$ will be non-empty at least on some non-trivial interval of time. Our estimates below show that in fact $\OO_t$ will be non-empty for all $t>0$.

 We will choose $\omega_0$ so that $\om_0=1$ on $\OO_0$ with smooth sharp (on a scale
$\lesssim \epsilon^{10}$) cutoff to zero into $D^{+}.$
This leaves some ambiguity in the definition of $\omega_0(x)$ away from $\OO_0$. We will see that it does not really matter how we define $\omega_0$ there, as long as we satisfy the conditions above.
For simplicity, one can think of $\omega_0(x)$ being just zero for $|x| < \delta$ away from a small neighborhood of $\OO_0$.
Using the estimates \eqref{velest1}, \eqref{velest2},
the estimate \eqref{mainterm111} and our choice of $\delta$ ensuring that $C_1 \log \delta^{-1} >> C(\gamma)$ we see that both $a$ and $b$ are decreasing functions of time and that near the diagonal $x_1 =x_2$ in $\{|x|<\delta\}$ we have
\begin{equation}\label{diag23} \frac{x_1(\log \delta^{-1}-C)}{x_2(\log \delta^{-1}+C)} \leq \frac{-u_1(x_1,x_2)}{\,\,\,\,u_2(x_1,x_2)} \leq \frac{x_1(\log \delta^{-1}+C)}{x_2(\log \delta^{-1}-C)}. \end{equation}
 This means that all particle trajectories for all times are directed into the $\phi > \pi/4$ region on the diagonal.
We claim that $\omega(x,t)=1$ on $\OO_t$. Indeed, it is clear that the ``fluid particles" which at $t=0$ are in $D^+\setminus\overline \OO_0$ cannot enter $\OO_{t'}$
through the diagonal $\{x_1=x_2\}$ due to~\eqref{diag23} at any time $0\le t'\le t$. Due to the very definition of $a(t), b(t)$ and $\OO_t$, they cannot enter $\OO_{t'}$ through the vertical segments $\{(a(t'), x_2)\in D^+\,,\,\,x_2<a(t')\}$ or
$\{(b(t'),x_2)\in D^+\,,\,\,x_2<b(t')\}$ at any time $0\le t'\le t$ either.  Finally, they obviously cannot enter through the boundary points of $D$. Hence the ``fluid particles" in $\OO_t$ must have been in $\OO_0$  at the initial time and we conclude that $\om(\,\cdot\, ,t)=1$ in $\OO_t$ from the Helmholtz law, or simply the vorticity equation $\om_t+(u \cdot \nabla) \om=0$.

By Lemma~\ref{mainlemma}, we have
\[ \underline u_1(b(t),t) \geq -b(t)\, \Omega(b(t), x_2(t)) - C \,b(t), \]
for some $x_2(t)\le b(t)\,,\,\,(x_2(t),b(t))\in  D^+$
as  $\|\omega(x,t)\|_{L^\infty} \leq 1$ by our choice of the initial datum $\om_0$.
A simple calculation shows that \[ \Omega(b(t), x_2(t)) \leq \Omega(b(t), b(t)) + C. \]
Indeed, since $x_2(t) \leq b(t)$ we can write
\begin{equation}\label{aux22}
 \int_{b}^2 \int_0^{b} \frac{y_1 y_2}{|y|^4}\,dy_2dy_1 =  {1\over 2}\int_{b}^2 y_1\left( \frac{1}{y_1^2} - \frac{1}{y_1^2+b^2} \right) \,dy_1 \leq b^2 \int_{b}^2 y_1^{-3}dy_1 \leq C.
\end{equation}
Thus we get
\begin{equation}\label{keyestCD}
\underline u_1(b(t),t) \geq -b(t)\, \Omega(b(t),b(t)) - C \,b(t).
\end{equation}

At the same time, for suitable $\tilde x_2(t)$ with $\tilde x_2(t)\le a(t)\,,\,\,(a(t),\tilde x_2(t))\in \bar D$ we have
\[ \overline u_1(a(t),t) \leq -a(t) \,\Omega(a(t), \tilde x_2 (t))+\tilde C a(t) \leq -a(t) \,\Omega(a(t), 0)+Ca(t), \]
by an estimate similar to \eqref{aux22} above.
Observe that
\[ \Omega(a(t), 0) \geq \frac{4}{\pi} \int_{\OO_t} \frac{y_1y_2}{|y|^4} \omega(y,t)\,dy_1dy_2 + \Omega(b(t), b(t)). \]
Since $\omega(y,t)=1$ on $\OO_t,$
\begin{eqnarray*}  \int_{\OO_t } \frac{y_1y_2}{|y|^4} \omega(y,t)\,dy_1dy_2 \geq  \int_{\epsilon}^{\pi/4} \int_{a(t)/\cos \phi}^{b(t)/\cos \phi}\frac{\sin 2\phi}{2r} \, dr d\phi > \\
 \frac18 (-\log a(t) + \log b(t) ) - C. \end{eqnarray*}
Therefore
\begin{equation}\label{keyestAB}
\overline u_1(a(t),t) \leq - \,a(t)\left({1\over 2\pi} (-\log a(t) + \log b(t) ) +\Omega(b(t), b(t)) \right) + C a(t).
\end{equation}

Note that from estimates \eqref{keyestCD}, \eqref{keyestAB} it follows that $a(t)$ and $b(t)$ are monotone decreasing in time, and by
finiteness of $\|u\|_{L^\infty}$ these functions are Lipschitz in $t$. Hence we have sufficient regularity for the following calculations.

\begin{eqnarray}\label{logCD}
\frac{d}{dt} \log b(t) \geq -\Omega(b(t), b(t)) - C\,, \\
\label{logAB}
\frac{d}{dt} \log a(t) \leq \frac{1}{2\pi}\left(\log a(t)-\log b(t)\right)- \Omega(b(t), b(t)) + C.
\end{eqnarray}

Subtracting \eqref{logCD} from \eqref{logAB}, we obtain
\begin{equation}\label{final1}
\frac{d}{dt} \left( \log a(t) - \log b(t) \right) \leq \frac{1}{2\pi} \left(\log a(t)-\log b(t)\right) +2C.
\end{equation}
From \eqref{final1}, the Gronwall lemma leads to
\begin{equation}\label{final2}
\log a(t) - \log b(t) \leq  \log \left( a(0)/ b(0)\right) \exp(t/2\pi) + C \exp(t/2\pi) \leq (9\log \epsilon +C) \exp(t/2\pi).
\end{equation}
We should choose our $\epsilon$ so that $-\log \epsilon$ is larger than the constant $C$ that appears in \eqref{final2}. In this case, we obtain from \eqref{final2} that
$\log a(t) \leq 8 \exp(t/2\pi) \log \epsilon,$ and so $a(t) \leq \epsilon^{8\exp(t/2\pi)}.$ Note that by the definition of $a(t)$ the first coordinate of the characteristic,
that originates at the point on $\partial D$ near the origin with $x_1= \epsilon^{10},$ does not exceed $a(t).$
To arrive at \eqref{main1}, it remains to note that
we can arrange $\|\nabla \omega_0\|_{L^\infty} \lesssim \epsilon^{-10}.$
\end{proof}

\noindent
{\it Remarks}
1. It is clear from the proof that the double exponential growth in our example is fairly robust. In particular,
it will be present for any initial data in a sufficiently small $L^\infty$ ball around $\omega_0$.
Essentially, what we need for the construction to work is symmetry, the dominance of $\Omega$ terms in Lemma~\ref{mainlemma},
and some additional structure of vorticity near the boundary.\\
 2. What can be said about the long-time behavior of the solutions away from the point $(0,0)$? This is a difficult question related to the general problem of the long-time behavior of 2d Euler solutions.
 Plausible conjectures can be made using methods of Statistical Mechanics, see for example~\cite{Miller1990, Robert1991, Shnirelman1993, MW2006}, in which one relies on conserved quantities and ergodicity-type
(mixing) assumptions (rather than the actual dynamics). Such assumptions are notoriously difficult to verify. In our situation this approach (conjecturally) predicts that as $t\to\infty$, the vorticity field $\om(x,t)$ should weakly$^*$ approach a steady-state solution, which - under our symmetry assumptions - can be expected to have a discontinuity along the axis of symmetry $\{x_1=0\}$. More specific predictions might depend on details of a particular model one would use. Most models would likely lead to a steady solution with two counter-rotating vortices, symmetric about the axis $\{x_1=0\}$. Deciding rigorously whether the actual dynamics will  follow such predictions seems to be beyond reach of existing methods.

\section{Discussion}\label{disc}

This paper proves that the double exponential in time upper bound on the growth of the vorticity gradient is sharp for the 2D Euler equation on a disk.
The growth in our example happens at the boundary. The scenario we provide is fairly robust and insensitive to small perturbations of the initial data within the same
symmetry class.
It looks likely that it should be possible to generalize our example yielding fast small scale creation to domains with smooth boundary
which possess an axis of symmetry, or to domains that lack a symmetry axis but have corners. The rate of growth may depend on the geometry of the domain.
Whether such examples can be constructed for an arbitrary domain with smooth
boundary is an interesting open question, though given the elements of robustness in the construction it is tempting to guess that exponential and double exponential growth
should be fairly common in domains with boundary. The question of whether double exponential growth of the vorticity gradient can happen in the bulk of the fluid in periodic or
full plane case also remains open. 

The example constructed in this paper has been inspired by numerical simulations of Tom Hou and Guo Luo \cite{HouLuo} who propose a new scenario for the development of a finite time
singularity in solutions to the 3D Euler equation at a boundary. Our construction here implies that double exponential in time growth in derivatives of vorticity is certainly possible
in 3D Euler; one does not need any growth in the amplitude of vorticity to achieve that. Any further growth in the 3D case must come from more complex nonlinear interactions, and
by the Kato-Beale-Majda criterion must involve infinite growth of vorticity. Controlling such solutions is a challenge.

\vspace{0.1cm}

\noindent {\bf Acknowledgement.} \rm AK acknowledges partial support of the NSF-DMS grants 1104415 and 1159133, and of Guggenheim fellowship.
VS acknowledges partial support of the NSF-DMS grants 1101428 and 1159376.
We would like to thank Tom Hou and Guo Luo for describing to us their numerical work, which led to the idea of the scenario presented here.
This work has been stimulated by the NSF-FRG Summer School on Partial Differential Equations at Stanford, August 2013, and the hospitality of Stanford University
is gratefully acknowledged. We would like to thank the
anonymous
 referees for many useful suggestions. 


\begin{thebibliography}{99}

\bibitem{BC} H. Bahouri, J.-Y. Chemin, \it \'Equations de transport relatives \'a des champs de vecteurs nonLipschitziens et m\'ecanique des
uides. (French) [Transport equations for non-Lipschitz vector
Fields and fluid mechanics], \rm  Arch. Rational Mech. Anal., {\bf 127} (1994), no. 2, 159--181

\bibitem{BKM}
J. T. Beale, T. Kato, and A. Majda,  \it Remarks on the breakdown of smooth solutions for the 3-D Euler equations, \rm
Commun. Math. Phys., {\bf 94}, pp. 61--66, 1984

\bibitem{Chemin} J.-Y. Chemin, \it Perfect Incompressible Fluids, \rm Clarendon press, Oxford, 1998

\bibitem{Den1} S. Denisov, Infinite superlinear growth of the gradient for the two-dimensional Euler equation,
Discrete Contin. Dyn. Syst. A, 23 (2009), no. 3, 755--764

\bibitem{Den2} S. Denisov, \it
Double-exponential growth of the vorticity gradient for the two-dimensional Euler equation, \rm  to appear in Proceedings of the AMS

\bibitem{Den3} S. Denisov, \it
 The sharp corner formation in 2D Euler dynamics of patches: infinite double exponential rate of merging, \rm preprint arXiv:1201.2210

\bibitem{Evans} L.C.~Evans, \it Partial Differential Equations, \rm Graduate Studies in Mathematics {\bf 19}, AMS, Providence, Rhode Island

\bibitem{GT} D.~Gilbarg and N.~Trudinger, \it Elliptic Partial Differential Equations of Second Order, \rm Springer-Verlag, Berlin Heidelberg 2001

\bibitem{Holder} E.~H\"older, \it \"Uber die unbeschr\"ankte Fortsetzbarkeit einer stetigen ebenen Bewegung in einer
unbegrenzten inkompressiblen Fl\"ussigkeit, \rm Math. Z. {\bf 37} (1933), 727?738

\bibitem{HouLuo} T.~Hou and G.~Luo, \it Potentially Singular Solutions of the 3D Incompressible Euler Equations, \rm preprint arXiv:1310.0497

\bibitem{Jud1} V. I. Judovic, \it The loss of smoothness of the solutions of Euler equations with time (Russian), \rm
Dinamika Splosn. Sredy, Vyp. {\bf 16}, Nestacionarnye Problemy Gidrodinamiki, (1974), 71--78,
121

\bibitem{Kato2} T.~Kato, \it On classical solutions of the two-dimensional non-stationary Euler equation, \rm
Arch. Rat. Mech. {\bf 27} (1968), 188--200

\bibitem{Kato} T.~Kato, \it
Remarks on the Euler and Navier-Stokes equations in R2.  \rm Nonlinear functional analysis and its applications, Part 2 (Berkeley, Calif., 1983), 1--7,
Proc. Sympos. Pure Math., {\bf 45}, Part 2, Amer. Math. Soc., Providence, RI, 1986

\bibitem{Koch} H.~Koch, \it Transport and instability for perfect fluids, \rm Math. Ann. {\bf 323} (2002), 491--523

\bibitem{MB} A.~Majda and A.~Bertozzi, \it Vorticity and
Incompressible Flow, \normalfont Cambridge University Press, 2002

\bibitem{MW2006} A.~Majda and X.~Wang, \it
Nonlinear Dynamics and Statistical Theories for Basic Geophysical Flows, \rm Cambridge University Press 2006.

\bibitem{MP}  C. Marchioro and M. Pulvirenti, \it Mathematical Theory of Incompressible Nonviscous Fluids, \rm
Applied Mathematical Sciences, {\bf 96}, Springer-Verlag, New York, 1994

\bibitem{Miller1990}
J.~Miller,
\it Statistical mechanics of {E}uler equations in two dimensions, \rm
 Phys. Rev. Lett.  65(17):2137--2140, 1990.

\bibitem{MSY} A.~Morgulis, A.~Shnirelman and V.~ Yudovich, \it Loss of smoothness and inherent instability of 2D inviscid fluid flows,
\rm  Comm. Partial Differential Equations {\bf 33} (2008), no. 4-6, 943-?968.

\bibitem{Nad1} N. S. Nadirashvili, Wandering solutions of the two-dimensional Euler equation, (Russian)
Funktsional. Anal. i Prilozhen., 25 (1991), 70--71; translation in Funct. Anal. Appl., 25 (1991),
220--221 (1992)

\bibitem{Robert1991} R.~Robert,
\it A maximum-entropy principle for two-dimensional perfect fluid
  dynamics.
\rm J. Statist. Phys. 65(3-4):531--553, 1991.

\bibitem{Shnirelman1993}
A.~I.~Shnirelman,
\it Lattice theory and flows of ideal incompressible fluid,
\rm Russian J. Math. Phys. 1(1):105--114, 1993.


\bibitem{Tao} T.~Tao, http://terrytao.wordpress.com/2007/03/18/why-global-regularity-for-navier-stokes-is-hard/
post by Nets Katz on 20 March, 2007 at 12:21 am and the following thread

\bibitem{Wolibner} W. Wolibner, \it Un theor\`eme sur l'existence du mouvement plan d'un
uide parfait, homog\`ene,
incompressible, pendant un temps infiniment long (French), \rm Mat. Z., {\bf 37} (1933), 698--726

\bibitem{Yud2} V. I. Yudovich, \it On the loss of smoothness of the solutions of the Euler equations and the
inherent instability of
ows of an ideal
fluid, \rm  Chaos, {\bf 10} (2000), 705--719

\bibitem{YudDE} V. I.~Yudovich, \it The flow of a perfect, incompressible liquid through a given region, \rm
Sov. Phys. Dokl. {\bf 7} (1963), 789--791

\end{thebibliography}
\end{document}